\documentclass[final,3p,times]{elsarticle}

\usepackage{amssymb}
\usepackage{amsmath}

\usepackage{enumerate}
\usepackage[dvipsnames]{xcolor}
\usepackage{graphicx}
\usepackage{hyperref}

\newtheorem{theorem}{Theorem}[section]
\newtheorem{lemma}[theorem]{Lemma}

\newtheorem{corollary}[theorem]{Corollary}
\newtheorem{proposition}[theorem]{Proposition}

\newproof{proof}{Proof}

\newcommand{\cN}{{\mathcal N}}
\newcommand{\cP}{{\mathcal P}}
\newcommand{\cR}{{\mathcal R}}

\newcommand{\cU}{{\mathcal U}} 

\journal{}

\begin{document}

\begin{frontmatter}

\title{Characterising rooted and unrooted tree-child networks}

\author[label1]{Janosch D\"ocker}
\ead{janosch.doecker@gmail.com}

\author[label1]{Simone Linz}
\ead{s.linz@auckland.ac.nz}

\address[label1]{School of Computer Science, University of Auckland, New Zealand}

\begin{abstract}
Rooted phylogenetic networks are used by  biologists to infer and represent complex evolutionary relationships between species that cannot be accurately explained by a phylogenetic tree.  Tree-child networks are a particular class of rooted phylogenetic networks that has been extensively investigated in recent years. In this paper, we give a novel characterisation of a tree-child network $\cR$  in terms of cherry-picking sequences that are sequences on the leaves of $\cR$ and reduce it to a single vertex by repeatedly applying one of two reductions to its leaves. We show that our characterisation  extends to unrooted tree-child networks which are mostly unexplored in the literature and, in turn, also offers a new approach to settling the computational complexity of deciding if an unrooted phylogenetic network can be oriented as a rooted tree-child network.
\end{abstract}


\end{frontmatter}

\section{Introduction}

Rooted phylogenetic networks play an increasingly important role in studying the evolution of biological entities (referred to as taxa) such as species, populations, and viruses whose evolutionary histories cannot be accurately represented by a phylogenetic tree~\cite{hibbins22}. From a graph-theoretic perspective, rooted phylogenetic networks (resp. trees) are leaf-labelled directed acyclic graphs (resp. leaf-labelled directed trees) with a single source vertex that is called the root. Typically, the leaves of rooted phylogenetic trees and networks correspond to taxa for which sequence data is available while all other vertices represent hypothetical ancestral taxa that are inferred. In a rooted phylogenetic tree, each non-leaf vertex can also be viewed as a speciation event, whereas in a rooted phylogenetic network each non-leaf vertex with in-degree one can be viewed as a speciation event and each non-leaf vertex with in-degree at least two can be viewed as a reticulation event. Reticulation events cause patterns of relationships that cannot be explained by a single phylogenetic tree and include processes like hybridisation, horizontal gene transfer, and recombination. They all have in common that genetic material is not exclusively transferred vertically but also horizontally between taxa.  Since the reconstruction and analysis of rooted phylogenetic networks is much more complex in comparison to rooted phylogenetic trees, many different classes of phylogenetic networks have been investigated in recent years (for a recent review, see~\cite{kong22}). While some of these classes are biologically motivated, most classes have been introduced because they have favourable algorithmic or mathematical properties. 

A particularly popular class of rooted phylogenetic networks is the class of tree-child networks. Roughly speaking (formal definitions are deferred to the next section), a rooted phylogenetic network is tree-child if no two vertices of in-degree at least two are joined by an edge or have a common parent. Tree-child networks were introduced in 2009~\cite{cardona09} and continue to be extensively investigated in contexts as diverse as counting and generating all tree-child networks, analysing collections of rooted phylogenetic trees that are simultaneously embedded in a tree-child network, and reconstructing tree-child networks from distance matrices, smaller subnetworks, or phylogenetic trees~ (e.g.,~\cite{bordewich18, cardona24,fuchs21,linz19,murakami19,semple16}). The wide acceptance of tree-child networks is due to the fact that their structural properties give enough mathematical traction to establish provable results and algorithms without being overly restrictive. In addition, several problems such as {\sc Tree-Containment}~\cite{iersel10} and {\sc Softwired-Parsimony}~\cite{frohn} that are NP-hard to solve for the encompassing class of all rooted phylogenetic networks become polynomial-time solvable or polynomial-time approximable within a small constant factor for tree-child networks.

Although it is desirable to directly reconstruct a rooted phylogenetic network from data in a one-step process, this is not always possible in practice. For instance, there can be uncertainties about the position of the root or about which vertex of a cycle represents a reticulation event. More generally, it can be challenging to infer the exact order in which reticulation and speciation events have occurred. In these situations, it can nonetheless  be possible to reconstruct an unrooted phylogenetic network first before  placing a root and assigning directions to the edges in a subsequent step to obtain a rooted phylogenetic network. As the name suggests, an unrooted phylogenetic network is a leaf-labelled undirected graph. While the definitions of some classes of rooted phylogenetic networks naturally carry over to the unrooted setting (e.g., level-$k$ and tree-based networks), unrooted tree-child networks remain largely unexplored in the literature. Nevertheless, the following related question  was recently raised in~\cite{huber22} and shown to be fixed-parameter tractable in the same paper. \\

\noindent{\sc Tree-Child-Orientation}\\
\noindent Given an unrooted phylogenetic network $\cU$ whose non-leaf vertices all have degree three, does there exist a rooted tree-child network $\cR$ such that $\cU$ can be obtained from $\cR$ by suppressing its root and omitting all edge directions? \\

\noindent Indirectly, the question suggests that the class of unrooted tree-child networks contains precisely those unrooted phylogenetic networks that can be oriented as a rooted tree-child network. Unfortunately, the computational complexity of {\sc Tree-Child-Orientation} and, therefore, of deciding if an unrooted phylogenetic network is tree-child remains open, despite considerable effort that, to date, has resulted in necessary conditions for when $\cU$ can be oriented as a rooted tree-child network~\cite{maeda23}, a practical parameterised algorithm as well as a fast heuristic to solve {\sc Tree-Child-Orientation}~\cite{urata24+}, and a complexity result that establishes NP-completeness for a variant of {\sc Tree-Child-Orientation} in which each non-leaf vertex of $\cU$ has degree five~\cite{doecker24+}.

Adopting the definition of unrooted tree-child networks suggested in the last paragraph, this paper establishes novel characterisations for rooted and unrooted tree-child networks, thereby offering a new way of attack to settle the computational complexity of {\sc Tree-Child-Orientation}. More specifically, we show that rooted as well as unrooted tree-child networks can be characterised by a certain type of cherry-picking sequence, that reduces a phylogenetic network to a smaller one by applying  one of two operations to its leaves until no further operation is possible. Cherry-picking sequences were introduced in the context of picking cherries from  phylogenetic trees to simultaneously embed a collection $\cP$ of rooted phylogenetic trees in a so-called temporal tree-child network while minimising the number of reticulation events that are necessary to explain $\cP$~\cite{humphries13}. Since then, the results presented in~\cite{humphries13} have been generalised to larger classes of rooted phylogenetic networks and to deciding if a given rooted phylogenetic network is embedded in another such network  (e.g.,~\cite{bordewich,janssen21,linz19}). Most recently, cherry-picking sequences have also been used in the context of computing distances between phylogenetic networks~\cite{landry23} and to develop practical algorithms that reconstruct phylogenetic networks from a collection of phylogenetic trees in a machine-learning framework~\cite{bernardini24}.  

The remainder of the paper is organised as follows. The next section  provides mathematical definitions and concepts that are used throughout the following sections. We then present characterisations for rooted and unrooted tree-child networks in terms of cherry-picking sequences in Section~\ref{sec:characterisation}. Section~\ref{sec:conclu} finishes  the paper with concluding remarks on stack-free networks and a type of cherry-picking sequence that is related to the one used to establish the results of Section~\ref{sec:characterisation}.

We end this section by noting that the idea of settling the  complexity of {\sc Tree-Child-Orientation} with the help of cherry-picking sequences was first mentioned in~\cite[page 214]{murakami21}, but not further pursued since then. 

\section{Preliminaries}\label{sec:prelim}

In this section, we introduce notation and terminology that is used in the remainder of the paper.\\

\noindent {\bf Phylogenetic networks.} Let $X$ be a non-empty finite set. An \emph{unrooted binary phylogenetic network} $\cU$ on $X$ is a simple undirected graph whose internal vertices all have degree 3 and whose leaves are bijectively labelled with the elements in $X$.  A \emph{rooted binary phylogenetic network} $\cR$ on $X$ is a directed acyclic graph with no loop and no parallel arcs that satisfies the following properties:
\begin{enumerate}[(i)]
\item there is a unique vertex $\rho$, the \emph{root}, with in-degree 0 and out-degree $2$,
\item a vertex of out-degree $0$ has in-degree 1 and the set of vertices with out-degree~$0$ is $X$, and
\item each internal vertex has either in-degree 1 and out-degree 2, or in-degree 2 and out-degree 1.
\end{enumerate}

\noindent For technical reasons, if $|X|=1$, then the isolated vertex labelled with the element in $X$ is an unrooted binary phylogenetic network and a rooted binary phylogenetic network, in which case this vertex is also considered to be the root.  Since all  rooted and unrooted phylogenetic networks are binary throughout this paper, we refer to a rooted binary phylogenetic network as a {\it rooted phylogenetic network on $X$} and to an unrooted binary phylogenetic network as an {\it unrooted phylogenetic network on $X$}. Moreover, whenever we use the term {\it phylogenetic network $\cN$ on $X$} without specifying whether $\cN$ is unrooted or rooted, then the following statement or definition applies to unrooted and rooted phylogenetic networks. For example, for a leaf $a$ of a phylogenetic network, we denote  the unique vertex adjacent to $a$ by $p_a$.

Let $\cR$ be a rooted phylogenetic network on $X$, and let $u$ and $v$ be two vertices of $\cR$. We say that $u$ (resp. $v$) is a {\it parent} (resp. {\it child}) of $v$ (resp. $u$) if $(u,v)$ is an arc in $\cR$. Moreover, $u$ is called a {\it tree vertex} if it has in-degree 1 and out-degree 2 and a {\it reticulation} if it has in-degree 2 and out-degree 1. Similarly, if an arc is directed into a reticulation, we refer to it as a {\it reticulation arc}. We define the {\it reticulation number} $r(\cR)$ of $\cR$ to be equal to the number of reticulations in $\cR$, and the  {\it reticulation number} $r(\cU)$ of an unrooted phylogenetic network $\cU$ with edge set $E$ and vertex set $V$ to be equal to $|E|-(|V|-1)$. Lastly, $\cR$ (resp. $\cU$) is called a {\it rooted} (resp. {\it unrooted}) {\it phylogenetic $X$-tree} if $r(\cR)=0$ (resp. $r(\cU)=0$).\\

\noindent {\bf Cherries and reticulated cherries.} Let $\cU$ be an unrooted phylogenetic network on $X$, and let $a$ and $b$ be two distinct elements in $X$. We say that $[a,b]$ is a {\it cherry} in $\cU$ if $p_a=p_b$, or $\cU$ consists of the single edge $\{a,b\}$. Furthermore, we say that $(a,b)$ is a {\it reticulated cherry} in $\cU$ if there exist two vertices $u$ and $v$ in $\cN$ such that $\{a,u\}$, $\{u,v\}$, and $\{v,b\}$ are edges of $\cU$ and $\{u,v\}$ is an edge of a cycle, in which case we refer to $\{u,v\}$ as the {\it reticulation edge} of $(a,b)$. Let $[a,b]$ be a cherry of $\cU$, and let $(c,d)$ be a reticulated cherry of $\cU$. Then {\it reducing} $[a,b]$  is the operation of deleting $a$ and, if $\cU$ consists of at least two edges, suppressing the resulting degree-2 vertex. Similarly, {\it reducing} $(c,d)$ is the operation of deleting the reticulation edge of $(c,d)$ and suppressing the two resulting degree-2 vertices. Collectively, we refer to these two operations as {\it cherry reductions}. Observe that a cherry reduction always results in an unrooted  phylogenetic network. 
 
Turning to rooted networks, let $\cR$ be a rooted phylogenetic network on $X$ with root $\rho$, and let $a$ and $b$ be two distinct elements in $X$. We say that $[a,b]$ is a {\it cherry} in $\cR$ if $p_a=p_b$. Furthermore, we say that $(a,b)$ is a {\it reticulated cherry} with reticulation leaf $a$ in $\cR$ if $p_a$ is a reticulation and $(p_b,p_a)$ is an arc in $\cR$. Let $[a,b]$ be a cherry of $\cR$, and let $(c,d)$ be a reticulated cherry of $\cR$ with reticulation leaf $c$. Then {\it reducing} $[a,b]$ is the operation of deleting $a$ and suppressing the resulting degree-2 vertex if $p_a\ne \rho$, and deleting $a$ and $p_a$ if $p_a=\rho$. Similarly, {\it reducing} $(c,d)$ is the operation of deleting the reticulation arc $(p_d,p_c)$ and suppressing the two resulting degree-2 vertices. As for unrooted phylogenetic networks, we collectively refer to the  operations of reducing $[a,b]$ and $(c,d)$ as {\it cherry reductions}. By definition of a reticulated cherry $(c,d)$ with reticulation leaf $c$, $p_d$ is a tree vertex in $\cR$. Thus, applying a cherry reduction to $\cR$ always results in a rooted phylogenetic network. To illustrate, the rooted phylogenetic network $\cR_0$ that is shown in Figure~\ref{fig:tree-child-def} has reticulated cherries $(b,a)$ and $(b,c)$ each with reticulation leaf $b$, and the rooted phylogenetic network $\cR_1$ that is shown in the same figure has cherry $[b,c]$.\\

\noindent {\bf Remark.} A phylogenetic network $\cN$ on $X$ has cherry $[a,b]$ if and only if it has cherry $[b,a]$. However, if $[a,b]$ is a cherry of $\cN$, then the operation of  reducing $[a,b]$ in $\cN$ results in a phylogenetic network on $X\setminus \{a\}$, whereas the operation of reducing $[b,a]$ results in a phylogenetic network on $X\setminus \{b\}$. On the other hand, an unrooted phylogenetic network $\cU$ has reticulated cherry $(a,b)$ if and only if it has reticulated cherry $(b,a)$ and the unrooted phylogenetic network obtained from $\cU$ by reducing $(a,b)$ is isomorphic to the one obtained from $\cU$ by reducing $(b,a)$. Lastly,  if $(a,b)$ is a reticulated cherry of a rooted phylogenetic network $\cR$, then the first coordinate of the ordered pair $(a,b)$ always refers to the reticulation leaf from now on.
Hence, for two elements $a,b\in X$, at most one of $(a,b)$ and $(b,a)$ is a reticulated cherry of $\cR$.  \\

\begin{figure}[t]
    \centering
\scalebox{0.95}{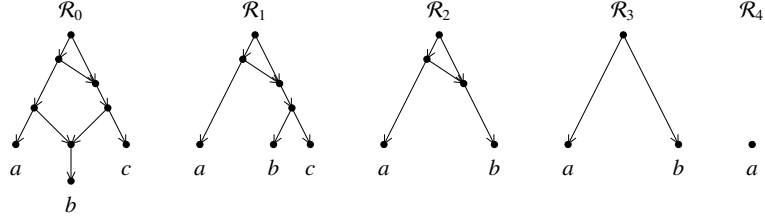}
    \caption{A complete cherry-reduction sequence $\sigma$ for the rooted phylogenetic network $\cR_0$  whose associated cherry-picking sequence is $((b,a),[c,b],(b,a),[b,a])$.}
    \label{fig:tree-child-def}
\end{figure}

\noindent{\bf Cherry-reduction sequences.} Let $\cN$ be a phylogenetic network on $X$, and  let $\sigma=(\cN=\cN_0,\cN_1,\cN_2,\ldots,\cN_k)$ be a sequence of  rooted or  unrooted phylogenetic networks such that, for each $i\in\{1,2,\ldots,k\}$, $\cN_i$ can be obtained from $\cN_{i-1}$ by a cherry reduction. We refer to $\sigma$ as a {\it cherry-reduction sequence} for $\cN$ and to $k$ as the {\it length} of $\sigma$. If $\cN_k$ has no cherry or reticulated cherry, then $\sigma$ is {\it maximal}. Moreover, if $\cN_k$ consists of a single vertex, then we say that $\sigma$ is {\it complete}. If $\sigma$ is complete, note that $k=|X|+r(\cN)-1$. In what follows, it is convenient to know which reductions have been applied in obtaining $\cN_k$ from $\cN_0$. We therefore associate a sequence $\Sigma=(r_1,r_2,\ldots,r_k)$ with $\sigma$ such that, for each $i\in\{1,2,\ldots,k\}$, $r_i=[x_i,y_i]$ if $\cN_i$ has been obtained from $\cN_{i-1}$ by reducing cherry $[x_i,y_i]$ and $r_i=(x_i,y_i)$ if $\cN_i$ has been obtained from $\cN_{i-1}$ by reducing reticulated cherry $(x_i,y_i)$. We refer to $\Sigma$ as a {\it cherry-picking sequence associated with $\sigma$} or, simply, as a {\it cherry-picking sequence} for $\cN$ if $\sigma$ is clear from the context or only plays an implicit role. Intuitively, each $r_i$ can be thought of as the reduction that has been applied to obtain $\cN_i$ from $\cN_{i-1}$. As for cherry-reduction sequences, we say that $\Sigma$ is {\it complete} if $k=|X|+r(\cN)-1$. For a cherry-reduction sequence $\sigma$ of a rooted phylogenetic network, the cherry-picking sequence associated with $\sigma$ is unique. On the other hand, for a complete cherry-reduction sequence $\sigma=(\cU=\cU_0,\cU_1,\cU_2,\ldots,\cU_k)$ of an unrooted phylogenetic network $\cU$, there are $2^{r(\cU)}$ cherry-picking sequences associated with $\sigma$ because, if $\cU_i$ is obtained from $\cU_{i-1}$ by reducing a reticulated cherry with leaves $a$ and $b$ for some $i\in\{1,2,\ldots,k\}$, then $r_i\in\{(a,b),(b,a)\}$. Figure~\ref{fig:tree-child-def} shows a complete cherry-reduction sequence $\sigma=(\cR_0,\cR_1,\cR_2,\cR_3,\cR_4)$ whose associated cherry-picking sequence is $((b,a),[c,b],(b,a),[b,a])$.  \\

\noindent{\bf Orchard and  tree-child networks.} We next define the classes of orchard~\cite{erdos19,janssen21,murakami21} and tree-child~\cite{cardona09} networks. Let $\cN$ be a  phylogenetic network on $X$. If $\cN$ has a complete cherry-reduction sequence, then $\cN$ is called an {\it orchard network}. As pointed out by Murakami~\cite[Figure 8.1]{murakami21}, not every maximal cherry-reduction sequence of an unrooted orchard network is  complete. This is in contrast to rooted orchard networks for which we have the following result~\cite[Proposition 4.1]{erdos19}.

\begin{proposition}\label{p:max}
Let $\cR$ be a rooted orchard network, and let $\sigma=(\cR=\cR_0,\cR_1,\cR_2,\ldots,\cR_k)$ be a maximal cherry-reduction sequence for $\cR$. Then $\sigma$ is complete.
\end{proposition}

\noindent While the last proposition implies that it can be decided in polynomial-time whether or not a rooted phylogenetic network is orchard, it was only shown most recently that it is NP-hard to decide if an unrooted phylogenetic network is orchard~\cite{dempsey}. 

A class of phylogenetic networks that is central to this paper, is the class of tree-child networks. We say that a  rooted phylogenetic network $\cR$ is {\it tree-child}, if every non-leaf vertex of $\cR$ has a child that is a tree vertex or a leaf. It is well known that a rooted tree-child network is also orchard, but not every rooted orchard network is also tree-child. Moreover, applying a cherry reduction to a rooted tree-child network always results in another rooted tree-child network~\cite[Lemma 4.1]{bordewich16}. In preparation for the next lemma, let $\cR$ be a rooted phylogenetic network. We say that $\cR$ contains a {\it stack} if there exist two reticulations that are joined by an arc and that $\cR$ contains a pair of {\it sibling reticulations} if there exist two reticulations that have a common parent. The following well-known equivalence now follows from the definition of a tree-child network and will be freely used throughout the remainder of the paper.

\begin{lemma}\label{l:equiv}
Let $\cR$ be a rooted phylogenetic network. Then $\cR$ is tree-child if and only if it has no stack and no sibling reticulations.
\end{lemma}

\begin{figure}[t]
    \centering
\scalebox{0.95}{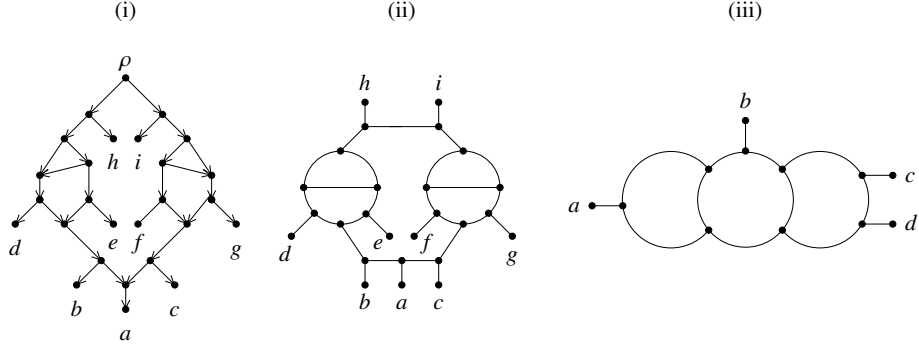}
    \caption{(i) A rooted tree-child network $\cR$ and (ii) an unrooted tree-child network $\cU$ such that $\cR$ is a tree-child orientation of $\cU$. (iii) An unrooted tree-child network for which there exist two complete cherry-picking sequences $\Sigma= ((d,c),(b,c),[c,d],(a,d),[a,d],[b,d])$ and $\Sigma'=((c,d),(b,c),[c,d],(a,d),[a,d],[b,d])$, but only $\Sigma$ is tree-child.}
    \label{fig:tree-child-order}
\end{figure}

Let $\cU$ be an unrooted  phylogenetic network on $X$, and let $\cR$ be a rooted  phylogenetic network on $X$ with root $\rho$. Following the terminology introduced in~\cite{huber22}, we say that $\cR$ is an {\it orientation} of $\cU$ if $\cU$ can be obtained from $\cR$ by replacing each arc with an (undirected) edge and suppressing $\rho$. It is worth noting that not every unrooted phylogenetic network has an orientation. For a characterisation of those unrooted phylogenetic networks that have an orientation, we refer the reader to~\cite[Lemma 4.13]{janssen18}. Now, $\cU$ is referred to as an unrooted {\it tree-child} network if it has an orientation that is a rooted tree-child network. If such an orientation exists, we call it a {\it tree-child orientation} of $\cU$. Figure~\ref{fig:tree-child-order}(i)--(ii) shows a rooted tree-child network $\cR$ and an unrooted tree-child network $\cU$ such that $\cR$ is a tree-child orientation of $\cU$.

\section{Cherry-picking sequences that characterise rooted and unrooted tree-child networks}\label{sec:characterisation}

In this section, we introduce a particular type of cherry-picking sequences that satisfy two additional properties. We then use these sequences to characterise rooted and unrooted tree-child networks in Sections~\ref{subsec:rooted} and~\ref{subsec:unrooted}, respectively. Let $\Sigma=(r_1,r_2,\ldots,r_k)$ be a cherry-picking sequence for a phylogenetic network. For each element $i\in\{1,2,\ldots,k\}$, recall that either $r_i=[x_i,y_i]$ or $r_i=(x_i,y_i)$. We say that $r_i$ {\it contains} $x_i$ and $y_i$. Let $j$ be the smallest element in $\{i+1,i+2,\ldots,k\}$ such that $r_j$ contains $x_i$. If such a $j$ exists, then we set $s(i)=j$ and $S(i)=r_{s(i)}$, in which case we call $S(i)$ the {\it successor pair} of $r_i$. If, on the other hand, $j$ does not exist, we set $s(i)=\infty$. With these definitions in hand, we say that $\Sigma$ is {\it tree-child} if the following two properties are satisfied. 
\begin{enumerate}[(P1)]
\item For each element $i\in\{1,2,\ldots,k\}$ such that $r_i=(x_i,y_i)$ and $s(i)=j$ for some $j\in\{i+1,i+2,\ldots,k\}$, we have $S(i)=[x_j,y_j]$.
\item For two distinct elements $i,j\in\{1,2,\ldots,k\}$ such that $r_i=(x_i,y_i)$ and $r_j=(x_j,y_j)$, and $s(i)\ne \infty$ and $s(j)\ne\infty$, we have $s(i)\ne s(j)$.
\end{enumerate}
\noindent Roughly speaking, (P1) requires that the successor pair of a reticulated cherry in $\Sigma$, if it exists, is a cherry, whereas (P2) requires that no element in $\Sigma$ is the successor pair of two distinct reticulated cherries. Referring back to Figure~\ref{fig:tree-child-def}, it is straightforward to check that the cherry-picking sequence that is associated with $\sigma$ is tree-child because $s(1)=2$ and $s(3)=4$.

\subsection{Rooted tree-child networks}\label{subsec:rooted}

We next characterise rooted tree-child networks in terms of tree-child cherry-picking sequences. To this end, we establish two lemmas.

\begin{lemma}\label{l:one}
Let $\cR$ be a rooted phylogenetic network on $X$. If there exists a tree-child cherry-picking sequence associated with a complete cherry-reduction sequence for $\cR$, then $\cR$ is tree-child.
\end{lemma}

\begin{proof}
Suppose that there exists a tree-child cherry-picking sequence $\Sigma=(r_1,r_2,\ldots,r_k)$  associated with a complete cherry-reduction sequence $\sigma=(\cR=\cR_0,\cR_1,\cR_2,\ldots,\cR_k)$ for $\cR$.  We show by induction on $k$ that $\cR$ is tree-child. As $\sigma$ is complete, $\cR_k$ is a single vertex. Thus, if $k=1$, then $\Sigma=([x_1,y_1])$ and the result follows because $\cR$ is the rooted tree-child network whose root is the parent of $x_1$ and $y_1$. Now assume that the result holds for all rooted phylogenetic networks that have a complete tree-child cherry-picking sequence whose length is less than $k$. Since $\sigma$ is a complete cherry-reduction sequence for $\cR$, it follows that $(\cR_1,\cR_2,\ldots,\cR_k)$ is such a sequence for $\cR_1$ whose associated cherry-picking sequence $(r_2,\ldots,r_k)$ is tree-child. Applying the induction assumption implies that $\cR_1$ is tree-child. If $r_1=[x_1,y_1]$, then $\cR_0$ is obtained from $\cR_1$ by subdividing the arc directed into $y_1$ with a new vertex $u$ and adding the arc $(u,x_1)$. As $\cR_1$ is tree-child, it follows that $\cR_0$ is tree-child. On the other hand, if $r_1=(x_1,y_1)$, then $\cR_0$ is obtained from $\cR_1$ by subdividing the arc directed into $y_1$ with a new vertex $u$, subdividing the arc directed into $x_1$ with a new vertex $v$, and adding the arc $(u,v)$. Towards a contradiction, assume that $\cR_0$ is not tree-child. Then $\cR_0$ contains either a stack or a pair of sibling reticulations. Let $w$ be the parent of $v$ that is not $u$ in $\cR_0$.

Assume that $\cR_0$ contains a stack. Since $\cR_1$ is tree-child and $u$ is a tree vertex in $\cR_0$, it follows that  $w$ is a reticulation.  Let $r_i$ with $i\in\{2,3,\ldots,k\}$ be the successor pair of $r_1$ in $\Sigma$, that is, $S(1)=r_i$. As $x_1$ is the child of $w$ in $\cR_1$, we have $r_i=(x_1,y_i)$, thereby contradicting that $\Sigma$ satisfies (P1). Hence, we may assume that $\cR_0$ contains a pair of sibling reticulations. As $\cR_1$ is tree-child, it follows that $w$ is the parent of two reticulations in $\cR_0$. Let $v'$ be the child of $w$ that is not $v$. Let $i$ be the smallest element in $\{2,3,\ldots,k\}$ such that $\cR_i$ does not contain $v'$. Then $r_i=(x_i,y_i)$, and no element in $\{r_{2},r_{3},\ldots,r_{i-1}\}$ contains $x_1$.  If  the arc $(w,v')$ is deleted in obtaining $\cR_i$ from $\cR_{i-1}$, then $y_i= x_1$, again contradicting that $\Sigma$ satisfies (P1). Hence, in obtaining $\cR_i$ from $\cR_{i-1}$, the arc $(w',v')$ is deleted, where $w'$ is the parent of $v'$ that is not $w$. This implies that $\cR_i$ has cherry $[x_1,x_i]$. But then $s(1)=s(i)$, thereby contradicting that $\Sigma$ satisfies (P2). It now follows that $\cR$ is tree-child, thereby establishing the lemma.\qed
\end{proof}

\begin{lemma}\label{l:two}
Let $\cR$ be a rooted phylogenetic network on $X$. If $\cR$ is tree-child, then each cherry-picking sequence associated with a complete cherry-reduction sequence for $\cR$ is tree-child.
\end{lemma}

\begin{proof}
Suppose that $\cR$ is tree-child. Let $\Sigma=(r_1,r_2,\ldots,r_k)$ be the unique cherry-picking sequence associated with a complete cherry-reduction sequence $\sigma=(\cR=\cR_0,\cR_1,\cR_2,\ldots,\cR_k)$  for $\cR$. Since $\cR$ is orchard,  $\sigma$ and, consequently, $\Sigma$ exist. Furthermore, since $\cR$ is tree-child, it follows from~\cite[Lemma 4.1]{bordewich16} that each rooted phylogenetic network in $\sigma$ is tree-child.  We show by induction on $k$ that $\Sigma$ is tree-child. Since the upcoming argument can be applied to  each cherry-picking sequence associated with a complete cherry-reduction sequence for $\cR$, the lemma then follows. For the upcoming argument, recall that $k=|X|+r(\cR)-1$. If $k=1$, then, as $\cR_1$ is a single vertex, we have $\Sigma=([x_1,y_1])$. The result follows.  Now assume that the result holds for all rooted  tree-child networks whose complete cherry-reduction sequences have length at most $k-1$. Consider the rooted tree-child network $\cR_1$ that has either one leaf or one reticulation less than $\cR$. Each complete  cherry-reduction sequence for $\cR_1$ has length $k-1$. Hence, by the induction assumption, each cherry-picking sequence  associated with a complete cherry-reduction sequence for $\cR_1$ is tree-child. In particular,  the complete cherry-picking sequence $(r_2,r_3,\ldots,r_k)$ for $\cR_1$ is tree-child. Clearly, if $r_1=[x_1,y_1]$, then $\Sigma$ is also tree-child. We may therefore assume that $r_1=(x_1,y_1)$. Let $v$ be the parent of $x_1$ in $\cR_0$. As $r_1=(x_1,y_1)$, $v$ is a reticulation. Let $w$ be the parent of $v$ in $\cR_0$ that is not the parent of $y_1$. Since $\cR_0$ is tree-child, $w$ is the root or a tree vertex. Let $r_i$ with $i\in\{2,3,\ldots,k\}$ be the successor pair of $r_1$, that is $S(1)=r_i$. Observe that $w$ is a vertex of $\cR_{i-1}$ and not a vertex of $\cR_i$. 

We next show that $\Sigma$ satisfies (P1). Clearly, this is the case if $r_i$ is a cherry pair. We may therefore assume towards a contradiction that  $r_i=(x_i,y_i)$. Since $w$ is not a reticulation in $\cR_0$, we have $y_i=x_1$. Hence, the child of $w$ that is not $x_1$ is a reticulation in $\cR_{i-1}$. But then the two children of $w$ in $\cR_0$ are both reticulations, a contradiction to $\cR_0$ being tree-child. It follows that $\Sigma$ satisfies (P1) and, in particular, $r_i$ is a cherry pair. We complete the proof by showing that $\Sigma$ also satisfies (P2). Again assume the contrary. Then there exists an element $j\in \{2,3,\ldots,i-1\}$ such that $s(1)=s(j)=i$ and $r_j=(x_j,y_j)$. Moreover, $\{x_1,x_j\}=\{x_i,y_i\}$. Let $v'$ be the reticulation in $\cR_0$ such that $x_j$ is the child of $v'$ in $\cR_{j-1}$, and let $p$ and $p'$ be the two parents of $v'$ in $\cR_0$ such that $(p,v')$ is deleted in obtaining $\cR_{j}$ from $\cR_{j-1}$.  Assume that there exists a directed path from $p'$ to  $x_1$ in $\cR_{j-1}$. Since $S(1)=r_i$ and $1<j<i$, it follows that $w=p'$. Thus, the two children of $w$ in $\cR_0$ are both reticulations, a contradiction. On the other hand, assume that there exists no directed path from $p'$ to $x_1$ in $\cR_{j-1}$. In this case, the successor pair of $r_j$ does not contain $x_1$, thereby contradicting that $s(j)=i$. It now follows that $\Sigma$ is tree-child. This completes the proof of the lemma. \qed
\end{proof}

The next theorem  follows from Lemmas~\ref{l:one} and~\ref{l:two} and  characterises rooted tree-child networks.

\begin{theorem}\label{t:characterisation}
Let $\cR$ be a rooted phylogenetic network on $X$. Then $\cR$ is tree-child if and only if  there exists a complete tree-child cherry-picking sequence for $\cR$.
\end{theorem}

\noindent {\bf Algorithmic consequences.} Let $\cR$ be a rooted phylogenetic network on $X$. Checking if $\cR$ is tree-child takes polynomial-time because one can simply check if $\cR$ has a stack or a pair of sibling reticulations. Lemmas~\ref{l:one} and~\ref{l:two} offer an alternative way of checking quickly if $\cR$ is tree-child. Essentially, one only has to compute a maximal cherry-reduction sequence $\sigma$ for $\cR$. By Proposition~\ref{p:max}, every maximal cherry-reduction sequence of a rooted orchard network is complete. Thus, if $\sigma$ is not complete, then $\cR$ is not orchard and, therefore, not tree-child. If $\sigma$ is complete but the cherry-picking sequence associated with $\sigma$ is not tree-child, then it follows from the contrapositive of Lemma~\ref{l:two} that $\cR$ is not tree-child. On the other hand, if $\sigma$ is complete and the cherry-picking sequence associated with $\sigma$ is tree-child, then $\cR$ is tree-child by Lemma~\ref{l:one}.

\subsection{Unrooted tree-child networks}\label{subsec:unrooted}

In this section, we characterise unrooted tree-child networks in terms of cherry-picking sequences. 

\begin{theorem} \label{t:tc-orientable}
Let $\cU$ be an unrooted phylogenetic network on $X$. Then  $\cU$ is tree-child if and only if there exists a complete tree-child cherry-picking sequence  for $\cU$.
\end{theorem}

\begin{proof}
First, suppose that $\cU$ is tree-child. Then there exists a rooted  tree-child network $\cR$ on $X$ that is an orientation of $\cU$.  Let $\sigma_\cR=(\cR=\cR_0,\cR_1,\cR_2,\ldots,\cR_k)$ be a maximal cherry-reduction sequence for $\cR$, and let $\Sigma=(r_1,r_2,\ldots,r_k)$ be the cherry-picking sequence associated with $\sigma_\cR$. It follows from  Proposition~\ref{p:max} and Lemma~\ref{l:two} , that $\Sigma$ is complete and tree-child. Recall that $k=|X|+r(\cR)-1$. We show by induction on $k$ that there exists a complete cherry-reduction sequence $\sigma_\cU=(\cU=\cU_0,\cU_1,\cU_2,\ldots,\cU_k)$ for $\cU$ such that $\Sigma$ is a cherry-picking sequence associated with $\sigma_\cU$. Since $\Sigma$ is tree-child, the result then follows. Assume that $k=1$. Then $r_1=[x_1,y_1]$ and $\cR$ is the rooted phylogenetic network whose root is adjacent to the two leaves $x_1$ and $y_1$. Since $\cR$ is an orientation of $\cU$, this implies that $\cU$ consists of the single edge $\{x_1,y_1\}$. Let $\cU_1$ be the unrooted phylogenetic network that can be obtained from $\cU$ by reducing $r_1$. Clearly $\sigma_\cU=(\cU=\cU_0,\cU_1)$ is a complete cherry-reduction sequence for $\cU$ whose associated cherry-picking sequence $\Sigma=(r_1)$ is tree-child. Now assume that $k>1$ and that the result holds for all unrooted tree-child networks for which a tree-child orientation exists whose complete cherry-reduction sequences have length less than $k$. Consider $r_1$. Since $\cR_0$ is a tree-child orientation of $\cU_0$, it follows that $\cU_0$ has either cherry $[x_1,y_1]$ if $r_1=[x_1,y_1]$ or reticulated cherry $(x_1,y_1)$ if $r_1=(x_1,y_1)$. Let $\cU_1$ be the unrooted phylogenetic network that can be obtained from $\cU_0$ by reducing $r_1$. As $\cR_0$ is a tree-child orientation of $\cU_0$ and $\cR_1$ is tree-child, it follows that $\cR_1$ is a tree-child orientation of $\cU_1$ and, in turn, $\cU_1$ is tree-child.  Furthermore, $\sigma_{\cR_1}=(\cR_1,\cR_2,\ldots,\cR_k)$ is a complete cherry-reduction sequence for $\cR_1$. Hence, by the induction assumption, there exists a complete cherry-reduction sequence $\sigma_{\cU_1}=(\cU_1,\cU_2,\ldots,\cU_k)$ such that $(r_2,r_3,\ldots,r_k)$ is a tree-child cherry-picking sequence associated with $\sigma_{\cU_1}$. By construction of $\cU_1$ from $\cU_0$, it follows in turn that $\Sigma$ is associated with $\sigma_\cU$. Thus, if $\cU$ is tree-child, then there exists a complete tree-child cherry-picking sequence for $\cU$.

Second, suppose that there exists complete tree-child cherry-picking sequence $\Sigma=(r_1,r_2,\ldots,r_k)$ for $\cU$. We show by induction on $k$ that there exists a rooted phylogenetic network $\cR$ and a complete cherry-reduction sequence $\sigma_\cR=(\cR=\cR_0,\cR_1,\cR_2,\ldots,\cR_k)$ for $\cR$ such that $\cR$ is an orientation of $\cU$ and $\Sigma$ is the cherry-picking sequence associated with $\sigma_\cR$. Since $\Sigma$ is tree-child, it then follows from Lemma~\ref{l:one} that $\cR$ is a tree-child orientation of $\cU$, thereby establishing the desired result. If $k=1$, then $r_1=[x_1,y_1]$. Let $\cR_0$ be the rooted phylogenetic network whose two leaves $x_1$ and $y_1$ are both children of the root, and let $\cR_1$ be the rooted phylogenetic network that can be obtained from $\cR_0$ by reducing $[x_1,y_1]$. Clearly $\cR_0$ is an orientation of $\cU_0$. Furthermore, $\sigma_\cR=(\cR=\cR_0,\cR_1)$ is a complete cherry-reduction sequence for $\cR$ whose associated cherry-picking sequence is $\Sigma=(r_1)$. Now assume that $k>1$ and that the result holds for all unrooted phylogenetic networks  that have a complete tree-child cherry-picking sequence of length less than $k$. Let $\cU_1$ be the unrooted phylogenetic network obtained from $\cU_0$ by reducing $r_1$. It follows that $\Sigma=(r_2,r_3,\ldots,r_k)$ is a complete tree-child cherry-picking sequence for $\cU_1$ of length $k-1$. By the induction assumption, there exists a rooted phylogenetic network $\cR_1$ and a complete cherry-reduction sequence $\sigma_{\cR_1}=(\cR_1,\cR_2,\ldots,\cR_{k})$ for $\cR_1$ such that $\cR_1$ is an orientation of $\cU_1$  and $(r_2,r_3,\ldots,r_{k})$ is the cherry-picking sequence associated with  $\sigma_{\cR_1}$. Now consider $r_1$. Assume that  $r_1=[x_1,y_1]$. Then obtain a rooted phylogenetic network $\cR_0$ from $\cR_1$ by subdividing the arc that is directed into $y_1$ with a new vertex $u$ and adding the arc $(u,x_1)$. Clearly, $[x_1,y_1]$ is a cherry of $\cR_0$. Since $\cU_1$ is obtained from $\cU_0$ by reducing $r_1$ and $\cR_1$ is an orientation of $\cU_1$, it follows that $\cR_0$ is an orientation of $\cU_0$. Moreover $(\cR=\cR_0,\cR_1,\cR_2,\ldots,\cR_{k})$ is a complete cherry-reduction sequence for $\cR_0$ whose associated cherry-picking sequence is $\Sigma$. We may now assume that  $r_1=(x_1,y_1)$.  Then obtain  a rooted phylogenetic network $\cR_0$ from $\cR_1$ by subdividing the arc that is directed into $x_1$ (resp. $y_1$) with a new vertex $v$ (resp. $u$) and adding the arc $(u,v)$. By construction $(x_1,y_1)$ is a reticulated cherry of $\cR_0$. As $(x_1,y_1)$ is also a reticulated cherry of $\cU_0$ and $\cR_1$ is an orientation of $\cU_1$ it again follows that $\cR_0$ is an orientation of $\cU_0$ and $(\cR=\cR_0,\cR_1,\cR_2,\ldots,\cR_{k})$ is a complete cherry-reduction sequence for $\cR_0$ whose associated cherry-picking sequence is $\Sigma$. Thus, if $\cU$ has a complete tree-child cherry-picking sequence, then $\cU$ is tree-child. This completes the proof of the theorem. \qed
\end{proof}

The next corollary is an immediate consequence of the first part of the proof of Theorem~\ref{t:tc-orientable}.

\begin{corollary}
Let $\cU$ be an unrooted phylogenetic network on $X$, and let $\cR$ be a tree-child orientation of $\cU$. Then $r(\cU)=r(\cR)$.
\end{corollary}

\noindent {\bf Algorithmic consequences.} As already mentioned in Section~\ref{sec:prelim}, not every maximal cherry-reduction sequence of an unrooted 
orchard network is complete. This somewhat negative result also extends to maximal cherry-reduction sequences of unrooted 
tree-child networks. More precisely, not every maximal cherry-reduction sequence of an unrooted tree-child network that is associated with a tree-child cherry-picking sequence is complete. For example, Figure~\ref{fig:tree-child-order}(ii) shows an unrooted tree-child network with reticulated cherry $(h,i)$ such that reducing $(h,i)$ results in an unrooted phylogenetic network that does not have any cherry or reticulated cherry. In addition and in contrast to Lemma~\ref{l:two}, Figure~\ref{fig:tree-child-order}(iii) shows an example of an unrooted tree-child network for which not every complete cherry-picking sequence is tree-child. Hence, Theorem~\ref{t:tc-orientable} does not (immediately) pave the way for a polynomial-time algorithm to decide whether or not an unrooted phylogenetic network $\cU$ is tree-child. This hurdle is a result of needing to check for an exponential number of cherry-picking sequences for $\cU$ whether or not they are complete and tree-child. Nevertheless, the characterisation as stated in Theorem~\ref{t:tc-orientable} offers a new approach to settle the computational complexity of  {\sc Tree-Child-Orientation}. In particular, instead of deciding if $\cU$ has a tree-child orientation, it is alternatively possible to check if $\cU$ has a  tree-child cherry-picking sequence associated with a complete cherry-reduction sequence for $\cU$.

\section{Concluding remarks}\label{sec:conclu}

In this section, we briefly discuss the implications of our work on the class of rooted stack-free networks and the differences between tree-child cherry-picking sequences as introduced in this paper and a related definition. We start by defining stack-free networks. A {\it rooted stack-free network} is a rooted phylogenetic network without any stack. The class of rooted stack-free networks intersects with the class of rooted orchard networks but is not contained in it. Thus, there exist rooted stack-free networks that are also orchard whereas others are stack-free but not orchard. Given the proof of Lemma~\ref{l:one}, it might be tempting to conjecture that rooted stack-free orchard networks can be characterised using  cherry-picking sequences that are only required to satisfy (P1) in the definition of a tree-child cherry-picking sequence. Unfortunately, this is not the case. Figure~\ref{fig:stack-free}(i) shows a rooted stack-free orchard network that is not tree-child and for which each complete cherry-picking sequence \ starts with $((b,a),(c,b))$ and does consequently not satisfy (P1). Thus (P1) and (P2) cannot be separated out in the sense that one avoids stacks and the other avoids sibling reticulations. 

In the context of simultaneously embedding a collection $\cP$ of rooted phylogenetic  $X$-trees in a rooted tree-child network whose number of reticulations is minimised, the authors of~\cite{linz19} defined a cherry-picking sequence for $\cP$ to be a sequence $([x_1,y_1],[x_2,y_2],\ldots,[x_k,y_k])$ of pairs  in $X$ that satisfies the following property.

\begin{enumerate}[(P)]
\item For all $i\in\{1,2,\ldots,k\}$, we have $x_i\notin\{y_{i+1}, y_{i+2},\ldots,y_k\}$. 
\end{enumerate}

\noindent Given two rooted tree-child networks $\cR$ and $\cR'$, we say that $\cR'$ is {\it contained} in $\cR$ if $\cR'$ can be obtained from $\cR$ by deleting reticulation arcs and suppressing resulting degree-2 vertices. Following the publication of~\cite{linz19}, Janssen and Murakami~\cite{janssen21}, subsequently showed that complete cherry-reduction sequences for $\cR$ whose associated cherry-picking sequences satisfy (P) are essentially sufficient to decide if $\cR'$ is contained in $\cR$. As such, it is interesting to ask if (P1) and (P2) in our definition of a tree-child cherry-picking sequence $(r_1,r_2,\ldots,r_k)$ can  be replaced with the following property while still obtaining a characterisation for tree-child networks.

\begin{enumerate}[(P3)]
\item For each element $r_i$ with $i\in \{1,2,\ldots,k\}$, its first coordinate $x_i$ is not equal to the second coordinate $y_j$ of any $r_j$ with $j\in\{i+1,i+2,\ldots,k\}$.
\end{enumerate}

\begin{figure}[t]
    \centering
\scalebox{0.95}{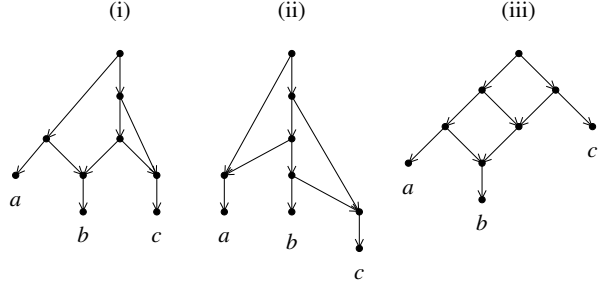}
    \caption{(i) A rooted stack-free orchard network for which no complete cherry-picking sequence satisfies (P1). (ii) A rooted tree-child network that has four complete cherry-picking sequences  and only one satisfies (P3). (iii) A rooted phylogenetic network that is not tree-child, but has a complete cherry-picking sequence  that satisfies (P3). See text for details.}
    \label{fig:stack-free}
\end{figure}

\noindent The answer however is that this is not possible. Figure~\ref{fig:stack-free}(ii) shows an example of a rooted tree-child network that has four complete cherry-picking sequences\\
 $$\Sigma_1=((c,b),(a,b),[b,c],[c,a]), \hspace{1cm} \Sigma_2=((c,b),(a,b),[b,c],[a,c]),$$
$$\Sigma_2=((c,b),(a,b),[c,b],[b,a]), \hspace{1cm}\Sigma_4=((c,b),(a,b),[c,b],[a,b]),$$
and only $\Sigma_4$ satisfies (P3). Moreover, Figure ~\ref{fig:stack-free}(iii) shows an example of a rooted phylogenetic network that is not tree-child but has a complete  cherry-picking sequence  $$((b,a),(b,a), [b,c],[a,c])$$  that satisfies (P3).\\

\noindent{\bf Acknowledgements.} We thank the New Zealand Marsden Fund for their financial support and Yukihiro Murakami for insightful discussions.

\end{document}